\RequirePackage{fix-cm}

\documentclass[smallextended]{svjour3}       % onecolumn (second format)
\smartqed  % flush right qed marks, e.g. at end of proof
\usepackage{graphicx}
\usepackage[colorlinks]{hyperref}
\usepackage{mathrsfs,enumitem,amsmath,amssymb}
\usepackage{tikz,tkz-euclide}
\usepackage{tikz-qtree}

\usetikzlibrary{calc,hobby}
\usetikzlibrary{decorations.pathmorphing,patterns,decorations,shapes,arrows, intersections,matrix,fit,calc,trees,positioning,arrows,chains, shapes.geometric,shapes,angles,quotes,fit,math}
\usetikzlibrary{arrows,positioning}
\usetikzlibrary{decorations.markings}

\def\ttitle{Effect of density dependence on coinfection dynamics}

\tikzset{
  big arrow/.style={
    decoration={markings,mark=at position 1 with {\arrow[scale=2,#1]{>}}},
    postaction={decorate},
    shorten >=0.4pt}}
\def\FIG{
\begin{figure}[h]
	\begin{center}
		\begin{tikzpicture}[every node/.style={ minimum height={1cm},minimum width={2cm},thick,align=center}]
		\node[draw] (I1) {$I_1$};
		\node[draw, above right=of I1] (S) {$S$};
		\node[draw, below=of S] (I12) {$I_{12}$};
		\node[draw, below=of I12] (R) {$R$};
		\node[draw, right= of I12] (I2) {$I_2$};
		\draw[->] (S) -- (I12) ;%node[above] {$\alpha_2$};
		\draw[->,] (I2) -- (I12);
		\draw[->,big arrow] (I1) -- (I12);
		\draw[->,big arrow] (I12) -- (R);
    \draw[->,big arrow] (S) .. controls (2,2) .. (I1) ;
    \draw[->,big arrow] (S) .. controls (4,2) .. (I2) ;

 \draw[->,big arrow] (I2) .. controls (5,-2) .. (R) ;
  \draw[->,big arrow] (I1) .. controls (1,-2) .. (R) ;

		\node[left =of I1] at (2.9,1.5) {$\alpha_1$};
		\node[left =of I1] at (7.2,1.5) {$\alpha_2$};
		\node[left =of I12]at (3.5,0.25) {$\eta_1,\gamma_1$};
		\node[left =of I1] at (5.3,1) {$\alpha_3$};
		\node[left =of I1] at (4.5,-0.9) {$\rho_3$};
		\node[left =of I1] at (6.5,0.25) {$\eta_2,\gamma_2$};
		\node[left =of I2] at (2.6,-1.8) {$\rho_1$};
		\node[left =of I1] at (7.5,-1.8) {$\rho_2$};
		
		\end{tikzpicture}
	\end{center}
	\caption{ Flow diagram for two strains coinfection model.}
	\label{flowdiag}
\end{figure}

}

\begin{document}

\title{\ttitle%\thanks{Grants or other notes
%about the article that should go on the front page should be
%placed here. General acknowledgments should be placed at the end of the article.}
}

%\subtitle{Do you have a subtitle?\\ If so, write it here}

%\titlerunning{Short form of title}        % if too long for running head

\author{Jonathan Andersson\and Samia Ghersheen \and Vladimir Kozlov \and Vladimir~G.~ Tkachev* \and Uno Wennergren}

\authorrunning{J. Andersson, S. Ghersheen, V. Kozlov, V. Tkachev, U. Wennergren} % if too long for running head

\institute{J. Andersson \at
              Department of Mathematics, Link\"oping University \\
              \email{jonathan.andersson@liu.se}           %  \\
%             \emph{Present address:} of F. Author  %  if needed
           \and
           S. Ghersheen\at
            Department of Mathematics, Link\"oping University \\
              \email{samia.ghersheen@liu.se}
              \and
            V. Kozlov\at
            Department of Mathematics, Link\"oping University \\
              \email{vladimir.kozlov@liu.se}
              \and
           V. Tkachev\at
            Department of Mathematics, Link\"oping University \\
              \email{vladimir.tkatjev@liu.se}
              \and
              U. Wennergren\at
Department of Physics, Chemistry, and Biology, Link\"oping University\\
\email{uno.wennergren@liu.se}
    }

\date{Received: date / Accepted: date}
% The correct dates will be entered by the editor

\maketitle

\begin{abstract}
In this paper we develop an SIR model for coinfection. We discuss how the underlying dynamics depends on the carrying capacity $K$: from a simple dynamics to a more complex. This can also help in understanding of appearance of more complicated dynamics,  for example, chaos and periodic oscillations, for large values of $K$. It is also presented that pathogens can invade in population and their invasion depends on the carrying capacity $K$ which shows that the progression of disease in population depends on carrying capacity. More specifically, we establish all possible scenarios (the so-called transition diagrams)  describing an evolution of an (always unique) locally stable equilibrium state for fixed fundamental parameters (transmission and death rates) as a function of the carrying capacity $K$. An important implication of our results is the following important observation. Note that one can regard the value of  $K$ as the natural `size' (the capacity) of a habitat. From this point of view, an isolation of individuals (the strategy which showed its efficiency for COVID-19 in various countries) into smaller resp. larger groups can be modelled by smaller resp. bigger values of $K$. Then we conclude that the infection dynamics becomes more complex for larger groups, as it fairly maybe expected for values of the reproduction number $R_0\approx 1$. We show even more, that for the values $R_0>1$ there are several (in fact four different) distinguished scenarios where the infection complexity (the number of nonzero infected classes) arises with growing $K$. Our approach is based on a bifurcation analysis which allows to generalize  considerably the previous Lotka-Volterra model considered previously in \cite{SKTW18a}.
% \PACS{PACS code1 \and PACS code2 \and more}
% \subclass{MSC code1 \and MSC code2 \and more}
\end{abstract}

	\section{Introduction}
	Two or more pathogens circulating in the same population of hosts can interact in various ways. One disease can, for instance, impart cross-immunity to the other, meaning that an individual infected with the first disease becomes partially or fully immune to infection with the second \cite{Castillo,Newman}. One disease can mediate the progression of other disease in population.
	
	Therefore it is important to understand the dynamics of coexistent  pathogens. In epidemiology the interaction of strains of the same pathogen, such as influenza or interacting diseases such as HIV/AIDS and hepatitis is very common and involves many complexities. The central problem in studying such systems is the explosive growth in the number of state variables of the system with the linear increase in the number of strains or pathogens \cite{Gog}. Mostly these strains or pathogens are interacting in a way which has limited the further understanding and dynamics of such systems in terms of limited analytical progress. In this regard, it is a challenge to understand the dynamics and evolution of pathogens in population. The complexity of multiple strain models allows a great variability in modelling strategies. Recently, attention has focused on understanding the mechanisms that lead to coexistence, competitive exclusion and co-evolution of pathogen strains in infectious diseases which is important from the management of disease prospective.

	Several studies exist on the coinfection with specific diseases.
	There are also studies \cite{May3,May4,Mosquera,Castillo,new_cite} which have addressed this issue in general. In \cite{Bremermann}, a mathematical model has been studied and showed that strains with differing degree of infectivity extinct, except for those that have higher the basic reproduction number. Allen et al in \cite{Allen2} showed coexistence only occur when the basic reproduction number is large enough for persistence of strains. They numerically illustrate the existence of globally stable coexistence equilibrium point. In an other study,  Allen et al  \cite{Allen}, studied an SI model of coinfection with application on hanta virus. They assumed a logistic growth with carrying capacity  and  horizontal transmission of both viruses and yet only vertical transmission of virus 2. The condition of coexistence of two strain is described.

	In \cite{Bichara},  a SIR model with vertical and horizontal transmission and a different population dynamics with limited immunity is considered.  It is shown that the competitive exclusion can occur which is independent of basic reproduction number but a threshold.  The existence and stability of endemic equilibrium is  also shown. Since coinfection involves many complexities, many studies are only restricted to numerical simulation to understand the dynamics.

	Nevertheless, mathematical modelling is one of the effective tool to understand the dynamics of biological system.  But the major challenge is to balance between the practicality and mathematical solvability of the model. The cost of realisticity in mathematical modelling is the diminution of mathematical machinery.

	The way to deal with this challenge is to divide the model into different sub models. Difference between the models is due to different biological assumptions. There are two major advantages in that case. First is the understanding of the system completely under certain assumptions. It can help to apply it to some real-life situations, since the controlling strategies for a diseases sometimes moves the original system to more simple system. In those cases the complete information about such system is needed to deal with that type of  unexpected situation from management prospective.  The second is, by relaxing assumptions, one can understand the role of each new parameter and its effects on the dynamics of epidemic.

	One of the important characteristics, to understand the coinfection dynamics, is transmission mechanism.
	In paper  \cite{SKTW18a} we have developed an SIR model to understand the dynamics of coinfection. Limited transmission is considered and the competitive exclusion principle is observed. The transition dynamics is also observed when the equilibrium points exist in the form of branches for each set of parameters. The compete dynamics of the system for all set of parameters is described by using linear complementarity problem. It appeared that there always exist an equilibrium point which is globally stable. It is showed that the dynamics of the system changes when carrying capacity changes.
	 There are certain assumptions on the transmission of coinfection in that model. It is assumed that the coinfection can only occur as a result of contact between coinfected class and susceptible class, coinfected class and single infected classes. Interaction between two single infected classes is not considered. Also the simultaneous transmission of two pathogens from coinfected individual to susceptible individual is assumed.

In this paper we develop an SIR model for coinfection which is a relevant extension of model presented in \cite{SKTW18a} to understand the role of each new transmission parameter in the dynamics. Our aim here is to investigate how the dynamics changes due to a certain parameter, which in our case is the carrying capacity $K$, from a simple dynamics to a more complicated. This can help in understanding of appearance of more complicated dynamics for example chaos etc. 	Contrary to \cite{SKTW18a}, we could no more make use of the linear complementarity problem due to some additional term which appeared by relaxing the assumption of interaction between two single infected classes. We use the approach based on bifurcation analysis. The density dependent population growth is also considered. It is presented that pathogens can invade in population and their invasion depends on the carrying capacity $K$ which shows that the progression of disease in population depends on carrying capacity.

\FIG
	
	%In section \ref{modelintro} we develop a model of coinfection under more realistic assumption on the transmission of pathogens. In section \ref{equipoint},\ref{stability} we discuss some basic properties of equilibrium points and their stability analysis.  In section \ref{transition},  \ref{sec:biff} we discuss the dependence of disease progression on the carrying capacity and examine different transition scenarios appear as a result of increasing carrying capacity and the bifurcation analysis.

\section{ Model formulation and the main result}

\subsection{The model}

The present model is displayed in Figure~\ref{flowdiag}. More precisely, we assume that the single infection cannot be transmitted by the contact with a coinfected person. According to Figure~\ref{flowdiag}, this process gives rise to the  system of ODEs:
\begin{equation}\label{submodel2}
\left\{
\begin{array}{ll}
S' &=(r(1-\frac{S}{K})-\alpha_1I_1-\alpha_2I_2-\alpha_3I_{12})S,\\
I_1' &=(\alpha_1S - \eta_1I_{12}-\gamma_1I_2 - \mu_1)I_1,\\
I_2' &=(\alpha_2S - \eta_2I_{12}-\gamma_2I_1- \mu_2)I_2,\\
I_{12}' &=(\alpha_3S+ \eta_1I_1+\eta_2I_2-\mu_3)I_{12}+\overline{\gamma}I_1I_2, \\
R' &=\rho_1 I_1+\rho_2I_2+\rho_3 I_{12}-d_4 R,
\end{array}
\right.
\end{equation}
where we use the following notation:
\begin{itemize}
\item[$\bullet$] $S$ represents the susceptible class,
\item[$\bullet$] $I_1$ and $I_2$ are the infected classes from strain 1 and strain 2 respectively,
\item[$\bullet$] $I_{12} $ represents the co-infected class,
\item[$\bullet$]
$R$ represents the recovered class.
\end{itemize}
Following \cite{Allen,Bremermann,Zhou}, we assume a limited population growth by making the per capita reproduction rate depend on the density of population. The recovery of each infected class is presented by the last equation in \eqref{submodel2}. The fundamental  parameters of the system are:
\begin{itemize}
\item[$\bullet$] $r=b-d_0$ is  the intrinsic rate of natural increase, where $b$ is the birthrate and $d_0$ is the death rate of $S$-class,
\item[$\bullet$]
$K $ is the carrying capacity (see also the next section),
\item[$\bullet$]
$\rho_i$ is the recovery rate from each infected class ($i=1,2,3$),
\item[$\bullet$] $d_i $ is the death rate of each class,  $( i=1,2,3,4)$, where $d_3$ and $d_4$ correspond $I_{12}$ and $R$ respectively,
\item[$\bullet$]
$\mu_i=\rho_i+d_i, i=1,2,3.$

\item[$\bullet$]
$\alpha_1$, $\alpha_2$, $\alpha_3$  are the rates of transmission of strain 1, strain 2 and both strains (in the case of coinfection),
\item[$\bullet$]
$\gamma_i$ is the rate at which infected with one strain get infected with  the other strain and move to a coinfected class ($i=1,2$),
\item[$\bullet$] $\bar\gamma=\gamma_1+\gamma_2$,

\item[$\bullet$]     $\eta_i$ is the rate at which infected from one strain getting  infection from a co-infected class $( i=1,2)$;
\end{itemize}

Summing up all equations in \eqref{initdata} we have
 \begin{equation}\label{sumup}
 \begin{array}{l}
N'=r(1-\frac{S}{K})S-d_1I_1-d_2I_2-d_3I_{12}-d_4 R
 \end{array}
 \end{equation}
 where
 $
 N=S+I_1+I_2+I_{12}+R
 $
 is the total population.
 %In particular, if the death mortalities $d_i$ are equal for each class, i.e. $d_i=d_0$ for all $i=1,2,3,4$, one obtains $N'=(b-\frac{rS}{K})S-d_0N$

We only need to consider the first four equations of \eqref{initdata} since $R$ appears only in the last equation, hence it does not affect the disease dynamics.
Rewrite the reduced system as
\begin{equation}\label{MAIN}
\left\{
\begin{array}{ll}
S' &=(r(1-\frac{S}{K})-\alpha_1I_1-\alpha_2I_2-\alpha_3I_{12})S \\
I_1' &=(\alpha_1S - \eta_1I_{12}-\gamma_1I_2 - \mu_1)I_1 \\
I_2' &=(\alpha_2S - \eta_2I_{12}-\gamma_2I_1- \mu_2)I_2 \\
I_{12}' &=(\alpha_3S+ \eta_1I_1+\eta_2I_2-\mu_3)I_{12}+\overline{\gamma}I_1I_2
\end{array}
\right.
\end{equation}

Furthermore, we only consider the case when the reproduction rate of the susceptible class is not less than their death rate, i.e.
$$
r>0 \quad \Leftrightarrow \quad b>d_0.
$$
Indeed, it is easy to see that the population will go extinct otherwise.
The reduced system is considered under the natural initial conditions
\begin{equation}\label{initdata}
S(0)>0,\quad I_1(0)\geq0, \quad I_2(0)\geq0,\quad  I_{12}(0)\geq0.
\end{equation}
Then it easily follows that any integral curve with \eqref{initdata} is well-defined and staying in the non negative cone for all $t\ge0$. Note also that since the variable $R$ is not present in the first four equations, without loss of generality, we may consider only the first four equations of system \eqref{submodel2}.
%It is convenient to keep the following unifying notation: $S=Y_0$, $I_1=Y_{1}$, $I_{2}=Y_{2}$ and $I_{12}=Y_{3}$.

\subsection{Reproduction rates}
It is convenient to introduce the notation
$$
\sigma_i:=\frac{\mu_i}{\alpha_i}, \qquad 1\le i\le 3.
$$
We shall always assume that the strains 1 and 2 are  different, i.e.
$ \sigma_1 \neq \sigma_2.
$
Then by change of the indices (if needed) we may assume that
$$\sigma_1 < \sigma_2.
$$
Under this assumption, $I_1$ is the primary disease.

Furthermore, let us assume for a moment that the susceptible class and \textit{only one} infected class are nonzero. Let us suppose that only $I_i$ is nonzero zero. Then \eqref{MAIN} reduces to
\begin{equation}\label{MAIN1}
\left\{
\begin{array}{ll}
S' &=(r(1-\frac{S}{K})-\alpha_iI_i)S \\
I_i' &=\alpha_i(S - \sigma_i)I_i
\end{array}
\right.
\end{equation}
It is easy to see that there always exist two equilibrium points: the trivial equilibrium $E_1=(0,0)$ and the  disease-free equilibrium $E_2=(K,0)$. Furthermore, if $K>\sigma_i$ then also exists (in the positive cone) the coexistence equilibrium $E_3=(\sigma_i, \frac{r}{\alpha_i}(1-\frac{\sigma_i}{K}))$.
Next, an elementary analysis shows that the following is true.

\begin{proposition}\label{pro:dim2}
The trivial equilibrium state $E_1$ is always unstable. For any positive $K\ne \sigma_1$ there exists a unique locally stable equilibrium point $E(K)$:
\begin{itemize}
  \item[$\bullet$] if $0<K<\sigma_i$ then $E(K)=E_2$;
  \item[$\bullet$] if $K>\sigma_i$ then $E(K)=E_3$.
\end{itemize}

\end{proposition}

The reproduction number
\begin{equation}\label{R0i}
R_0(I_i):=\frac{K}{\sigma_i}
\end{equation}
can be used as a threshold. In other words, the transition from the disease-free equilibrium state to the disease equilibrium (the coexistence equilibrium point) occurs exactly when the reproduction number $R_0(I_i)$ of the corresponding infected class $I_i$  exceeds $1$. We illustrate the transition by the diagram
$$
E_2 \rightarrow E_3.
$$

The latter also clarifies the meaning of the parameter $\sigma_i$.  Namely, note that a more aggressive virus $I$ has a greater value of $R_0(I)$. For a fixed value of the carrying capacity $K$ this implies that a more aggressive virus $I$ has a smaller value of $\sigma$ (which, for example, means smaller recovery rate $\rho$ or greater rate of transmission $\alpha$).

It is natural to assume that the reproduction number of coinfection must be less than that of virus 1 and 2 respectively \cite{Martcheva}. This  makes it natural to assume the following hypotheses:
\begin{equation}\label{assum}
\sigma_1< \sigma_2 < \sigma_3.
\end{equation}

\subsection{Some important notation}
We shall also assume that \eqref{MAIN} satisfies the following non-degenerate condition
\begin{equation}\label{deltamu}
\Delta_\alpha=\left|
              \begin{array}{cc}
                \eta_1 & \alpha_1 \\
                \eta_2 & \alpha_2 \\
              \end{array}
            \right|=\eta_1\alpha_2-\eta_2\alpha_1\neq 0.
\end{equation}
This condition has a natural  biological explanation: the virus strains 1 and 2 have different (co)infections rates.
Let us also define
\begin{align}
A_1&= \frac{\alpha_1\alpha_3}{r}(\sigma_3-\sigma_1),\qquad \eta_1^*:=\frac{\eta_1}{A_1}\label{A1}\\
  A_2&= \frac{\alpha_2\alpha_3}{r}(\sigma_3-\sigma_2),\qquad \eta_2^*:=\frac{\eta_2}{A_2}\label{A3}\\
  A_3&= \frac{\alpha_1\alpha_2}{r}(\sigma_2-\sigma_1), \qquad \gamma^*:=\frac{\gamma_1}{A_3}.\label{A2}
 \end{align}
 By \eqref{assum} $A_1,A_2,A_3>0$.
 We also have
  \begin{equation}\label{etaA123}
  \alpha_2A_1=\alpha_3A_3+  \alpha_1A_2
  \end{equation}
and
\begin{equation}\label{deltamu1}
\Delta_\mu=\frac{\eta_1r}{\alpha_1}A_3+\sigma_1\Delta_\alpha=
\frac{\eta_2r}{\alpha_2}A_3+\sigma_2\Delta_\alpha,
\end{equation}
hence $A_3>0$ implies
\begin{equation}\label{deltamu2}
\Delta_\mu>\sigma_1\Delta_\alpha \qquad \Delta_\mu>\sigma_2\Delta_\alpha.
\end{equation}
This implies an inequality which will be useful in the further analysis:
\begin{equation}\label{deltamu21}
\sigma_2(\Delta_\alpha+\gamma_2\alpha_3)<\Delta_\mu+\gamma_2\mu_3.
\end{equation}
%and for small enough $\gamma_1$ we even have
%\begin{equation}\label{deltamu21}
%\sigma_2(\Delta_\alpha-\gamma_1\alpha_3)<\Delta_\mu-\gamma_1\mu_3
%\end{equation}
We shall further make use of the following relations:
 \begin{equation}\label{etaA13}
\begin{split}\eta_1 A_2 -\eta_2A_1&< \eta_1\frac{\alpha_2}{\alpha_1}A_1-\eta_2A_1=\Delta_\alpha \frac{A_1}{\alpha_1}.
\end{split}
\end{equation}
On the other hand, one has
\begin{align}\label{etaAetaA}
\eta_2^*- \eta_1^* =
\frac{(\Delta_\mu-\Delta_\alpha\sigma_3)\alpha_3}{A_1A_2r}
\end{align}

\begin{remark}
The parameters $\eta_i^*$ can be thought of as the normalized co-infection rates. They play a distinguished role in the analysis of the thresholds given below.
\end{remark}

\subsection{The carrying capacity}
The concise meaning of the parameter $K $ becomes clear if we consider the limit case of \eqref{MAIN} when the virus infection is absent, i.e. $I_1=I_2=I_{12}=0$. Then \eqref{submodel2} reduces to the system
\begin{eqnarray}
\label{Verhulst}
S'&=&r(1-\frac{S}{K})S\\
R'&=&-\mu_4'R,\label{Req}
\end{eqnarray}
where the first equation \eqref{Verhulst} is the famous logistic (Verhulst) equation, $r$ is  the \textit{intrinsic rate of natural increase} and
$K$ is the \textit{carrying capacity} of the system.  The carrying capacity $K$ is one of the most fundamental parameters in population dynamics and it usually expresses the upper limit on the size of hypothetical populations, thereby enhancing mathematical stability. In basic ecology one defines carrying capacity as the equilibrium population size. Indeed, coming back to \eqref{MAIN}, we can see that $K$ coincides with the healthy population size for the \textit{disease-free equilibrium}. Mathematically this means that for any positive initial data, the corresponding  solution of \eqref{Verhulst} converges to $K$ as $t\to \infty$. Furthermore, the equilibrium state $G_2:=(K, 0,0,0)$ is the only possible equilibrium point of \eqref{MAIN} with all $I_i=0$.

\subsection{The main result}
Equilibrium points of \eqref{MAIN} are determined by the system
\begin{equation}\label{Equilib}
\begin{split}
(r(1-\frac{S}{K})-\alpha_1I_1-\alpha_2I_2-\alpha_3I_{12})S=0,\\
(\alpha_1S - \eta_1I_{12}-\gamma_1I_2 - \mu_1)I_1=0,\\
(\alpha_2S - \eta_2I_{12}-\gamma_2I_1- \mu_2)I_2=0,\\
(\alpha_3S+ \eta_1I_1+\eta_2I_2-\mu_3)I_{12}+\overline{\gamma}I_1I_2=0.
\end{split}
\end{equation}
It is an elementary to see (see also Proposition~\ref{pro:equil} below for more explicit representations) that except for the trivial equilibrium point
$$G_1=(0,0,0,0)$$ and the disease-free equilibrium $$G_2=(K, 0,0,0),
$$
there exist only $6$ possible equilibrium points:

\begin{itemize}
\item[$\bullet$]
three semi-trivial equilibria $G_3,G_4,G_5$ with \textit{only one nonzero infected class}, i.e. $I_{i}\ne0$ for some $i$;
\item[$\bullet$]
two \textit{coinfected semi-trivial equilibria} $G_6,G_7$ with $I_{12}\ne0$ but $I_1I_2=0$;
\item[$\bullet$]
the \textit{coexistence equilibrium} $G_8$ with $SI_1I_2I_{12}\ne0$.
\end{itemize}

Our main result extends the results obtained in \cite{SKTW18a} on the case of arbitrary values of $\gamma_i$. More precisely, we will prove that we have the following possible scenarios for developing of an equilibrium point as a continuous function of increasing carrying capacity $K$:

\begin{theorem}\label{theo1}
Let us assume that
 \begin{equation}\label{etaeq}
0<\eta_1^*< \max\{1, \eta_2^*\}.
 \end{equation}
 Then there is exactly one locally stable nonnegative equilibrium point. Furthermore, changing the carrying capacity $K$, the type of this locally stable equilibrium point may be exactly  one of the following alternative cases:
\begin{enumerate}[label=(\roman*),itemsep=0.1ex,leftmargin=0.8cm]
        \item \label{it1}for $\eta_1^*< 1$ one has $G_2\rightarrow G_3$.
        More precisely,
        \smallskip
 \begin{itemize}
 \item[$\bullet$] if $0< K<\sigma_1$  then $G_2$ is locally stable;
  \item[$\bullet$] if $K>\sigma_1$ then $G_3$ is  locally stable.
 \end{itemize}
\smallskip
        \item \label{it4}for $1<\eta_1^*<\eta_2^*$ one has $G_2\rightarrow G_3\rightarrow G_6\rightarrow G_5$. More precisely,

        \smallskip
  \begin{itemize}
 \item[$\bullet$] if $0< K< \sigma_1$  then $G_2$ is locally stable;
  \item[$\bullet$] if $\sigma_1<K<  K_1$ then $G_3$ is locally stable;
  \item[$\bullet$] if $K_1<K<  K_2$ the point $G_6$ is locally stable;
  \item[$\bullet$] if $K> K_3$ then the point $G_5$ is locally stable
 \end{itemize}
 where
 $$
  K_1=\frac{\sigma_1\eta_1^*}{\eta_1^*-1},\qquad K_2=\frac{\sigma_3}{\sigma_1}K_1.
  $$

\end{enumerate}
\end{theorem}

\begin{remark}
We consider the remained case
$$
\eta_1^*> \min\{1, \eta_2^*\}
%\eta_1\ge A_1, \qquad \eta_2A_1< \eta_1 A_2.
$$
in the forthcoming paper \cite{forthcoming}. This requires a delicate bifurcation analysis with application of methods similar to the principle of the exchange of stability developed in \cite{CrandallRabinowitz1}; see also \cite{DiekmannGetto} and \cite{Boldin} for recent applications in population analysis.
We will show that in the remained cases one has the following two transition diagrams:
\begin{enumerate}[label=(\roman*),itemsep=0.1ex,leftmargin=0.8cm]
            \item[(iii)] \label{it2} $G_2 \rightarrow G_3\rightarrow G_6\rightarrow G_8\rightarrow G_7\rightarrow G_5$;
        \item[(iv)] \label{it3}$G_2\rightarrow G_3\rightarrow G_6\rightarrow G_8$.
    \end{enumerate}
    Furthermore, $G_8$ may loose stability for large $K$ and small $\gamma_i$ in the latter case.
\end{remark}

\begin{remark}
In particular, the above result implies that there are only three possible `final destination' equilibrium states, namely $G_3, G_5$ and $G_8$. The latter may be interpreted as follows: taking the carrying capacity $K$ as a random variable in $(0,\infty)$, one concludes that a random choice of $K$ corresponds to `large enough' values of $K$, thus implying that, in average, the system \eqref{submodel2} will have either of the three of the possible finishing scenarios $G_3, G_5$ and $G_8$. This is interesting from the epidemiological point of view: this means that if $R_0>1$, the \textit{most expected equilibrium states} (in the aforementioned meaning) is either of the following: the state with the presence of only the first strain $G_3$, purely coinfected case $G_5$, or the state with the presence of all possible strains.
\end{remark}

\section{Basic properties of equilibrium points}\label{equipoint}
First we discuss some general results and equilibrium point analysis for \eqref{submodel2}.
%\begin{lemma}
%	If $S \neq 0$ and $Y_1,Y_2,Y_{12} \neq 0$ is an elements of
%\end{lemma}

\subsection{A priori bounds}
In this section we discuss only stable equilibrium points with nonnegative coordinates. We denote
$$
Y=(S, I_1,I_2,I_{12}).
$$
In what follows, by an \textit{equilibrium point} we always mean an equilibrium $Y$ of \eqref{MAIN} point with nonnegative coordinates, $Y=(S, I_1,I_2,I_{12})\ge0$.

In the next sections we identify all equilibria of the system \eqref{MAIN}
and determine their local stability properties. First, let us remark some useful  relations which hold for \textit{any} nonnegative equilibrium point of \eqref{MAIN}.

\begin{lemma}\label{lem:equi}
Let $Y=(S,I_1,I_2,I_{12})\ne (0,0,0,0)$ be a nontrivial equilibrium point of \eqref{MAIN} with nonnegative coordinates. Then
\begin{equation}\label{Ystar}
	0<S\le K,
\end{equation}
and the right equality holds if and only if $I_1=I_2=I_{12}=0$, i.e. precisely when
$$
Y=G_2:=(K,0,0,0).
$$
Furthermore,
\begin{equation}\label{zineq}
	\sigma_1\le S\le \min\{K,\sigma_3\},
\end{equation}
unless $Y=G_2$.

\end{lemma}

\begin{proof}
Let $S=0$. Then we have from the second equation of \eqref{Equilib} that $(\eta_1I_{12}+\gamma_1I_2 + \mu_1)I_1=0$, where the nonnegativity assumption gives $\eta_1I_{12}+\gamma_1I_2 + \mu_1\ge \mu_1>0$, hence  $I_1=0$. For the same reason, $I_2=0$, thus the last equation in \eqref{Equilib} yields
$\mu_3I_{12}=0, $  hence $I_{12}=0$ too. This proves that $Y=(0,0,0,0)$, hence implying the left inequality in \eqref{Ystar}.

Now assume that $Y=(S,I_1,I_2,I_{12})\ne (0,0,0,0)$ is an equilibrium point. Since $S\ne 0$, we have from the first equation of \eqref{Equilib}  that
\begin{equation}\label{law1}
\alpha_1I_1+\alpha_2I_2+\alpha_3I_{12}=\frac{r(K-S)}{K}.
\end{equation}
In particular, the nonnegativity of the left hand side in the latter identity implies that $K-S\ge0$, i.e. proving the right inequality in \eqref{Ystar}.
On the other hand, summing up  equations in \eqref{Equilib} we obtain
	\begin{equation}
	\mu_1I_1+\mu_2I_2+\mu_3I_{12}=\frac{r(K-S)S}{K}.
	\label{law2}
	\end{equation}
Assuming that $S\ne K$ and dividing  \eqref{law2} by  \eqref{law1} we get $$S=\frac{\mu_1I_1+\mu_2I_2+\mu_3I_{12}}{\alpha_1I_1+\alpha_2I_2+\alpha_3I_{12}}
$$
which readily yields \eqref{zineq}.
\end{proof}

This implies, in particular

\begin{corollary}\label{cor:Sstar}
 For any equilibrium point $Y\neq(0,0,0,0)$ and $Y\neq G_2$ there holds $K\geq \sigma_1$.
\end{corollary}

Notice that for $G_2$, all $I_i=0$, otherwise we have

\begin{corollary}

If an equilibrium point $Y$ is distinct from $G_2:=(K,0,0,0)$ then \eqref{law1} implies the following a priori bound on the $I$-coordinates:
	\begin{equation}\label{max}
	\sigma_1\le S\le \sigma_3, \qquad 0\le I_i\le \frac{r}{\alpha_i}, \qquad i=1,2,3,
	\end{equation}
where $r=r$ is  the intrinsic rate of natural increase.
In other words, any equilibrium point distinct  from $G_2$ lies inside a block with sides depending only on the fundamental constants.
\end{corollary}

%\subsection{Stability of equilibrium points}
%A closer analysis reveals that the reduced system \eqref{MAIN} has one trivial equilibrium point (the origin), seven boundary equilibrium points and a coexistence equilibrium point.
The trivial equilibrium point  $G_1=(0,0,0,0)$ is the equilibrium of no disease or susceptible and the standard (local asymptotic) stability treatment  shows that this point is always unstable.
The first nontrivial equilibrium point $G_2$ is {the} \textit{disease-free equilibrium}, i.e
$$%\begin{equation}\label{subsdstar}
G_2=(K, 0,0,0)
$$%\end{equation}
and it always exist (for any admissible values of the fundamental parameters).  The argument of \cite{SKTW18a} is also applicable in the present case because the stability analysis for $G_2$ does not involve $\gamma_i$, so it is literally equivalent to that given in \cite{SKTW18a}. Repeating this argument (see section~8 in \cite{SKTW18a}) readily yields the following criterium.

\begin{proposition}\label{pro:G2}
The following three conditions are equivalent:
\begin{enumerate}[label=(\alph*),itemsep=0.4ex,leftmargin=0.8cm]
\item
the disease-free equilibrium point $G_2$ is locally stable;
\item
the disease-free equilibrium point $G_2$ is globally (asymptotically) stable;
\item
$0< K<\sigma_1.$
\end{enumerate}

	\end{proposition}

\begin{remark}
The latter proposition is completely consistent with the dichotomy of the $R_0$-number (the reproduction number, sometimes called basic reproductive ratio). Recall that in epidemiology, the basic reproduction number of an infection can be thought of as the number of cases one case generates on average over the course of its infectious period, in an otherwise uninfected population. In our case, using the formal definition (see for example \cite{Diekmann1990}), one has
$$
R_0=\max\{\frac{K}{\sigma_i}:1\le i\le 3\}=\frac{K}{\sigma_1},
$$
using the fact that the first strain is the most inclined to spread.

In this notation, $R_0<1$ corresponds exactly to the scenario when the infection will die out in the long run (i.e. the only asymptotically stable equilibrium state is the disease-free equilibrium point $G_2$), while $R_0>1$ means the infection will be able to spread in a population. Therefore, in what follows, we shall focus on the nontrivial case $R_0>1$ with different scenario admitting the equilibrium states with some of $I_1,I_2,I_{12}$ nonzero.
\end{remark}

\subsection{Explicit representations of equilibrium points}

Coming back to \eqref{Equilib}, note that the Bezout theorem yields (in generic setting) that a quadratic system with four equations and four independent variables has  $2^{4}=16$ distinct solutions (counting the identically zero solution $(0,0,0,0)$). In fact, in our case we have only one-half of the relevant  (the Bezout number) solutions. More precisely, we have

\begin{proposition}\label{pro:equil}
Except for the trivial equilibrium $G_1=(0,0,0,0)$ and the disease-free equilibrium $G_2=(K, 0,0,0)$ there exist only the following  equilibrium states:
\begin{align}\label{subeqpt3}
G_3  &=\left(\sigma_1, I_1, 0,0\right),\quad I_1 =\frac{r}{\alpha_1}(1- \frac{\sigma_1}{K}),\\
\label{G4coor}
G_4& =(\sigma_2,0,I_2, 0),\quad I_2:=\frac{r}{\alpha_2}(1- \frac{\sigma_2}{K}),\\
\label{subeqpt5}
G_5 &=(\sigma_3,0,0,I_{12})\quad I_{12}=\frac{r}{\alpha_2}(1- \frac{\sigma_3}{K}),\\
\label{subeqpt6}
G_6&=(S,I_1,0,I_{12}), \,\,\, S=\frac{\sigma_1K}{K_1}, \,\,\,
I_1=\frac{\mu_3}{\eta_1}(1-\frac{K}{K_2}),\,\,\,%\nonumber\\
I_{12}=\frac{\mu_1}{\eta_1}(\frac{K}{K_1}-1),\\
G_7&=(S,0,I_2,I_{12}), \,\,\,
S=\frac{\sigma_1K}{K_3},\,\,
I_2=\frac{\mu_3}{\eta_2}(1-\frac{K}{K_4}),\,\,\,
I_{12}=\frac{\mu_2}{\eta_2}(\frac{K}{K_3}-1),\label{subeqpt7}\\
G_8&=(S,I_1,I_2,I_{12}),\label{subeqpt8}
\end{align}
where
 $$
  K_3=\frac{\sigma_2\eta_2^*}{\eta_2^*-1},\qquad K_4=\frac{\sigma_3}{\sigma_2}K_3.
  $$
  and there may exist at most two distinct points of type $G_8$.
\end{proposition}

\begin{proof}
Let $Y=(S,I_1,I_2,I_{12})\ne G_1, G_2$ be an equilibrium point. Then by Lemma~\ref{lem:equi} $S>0$ and by the assumption some of coordinates $I_1,I_2,I_{12}$ must be distinct from zero. First assume that $I_{12}=0$. Then the last equation in \eqref{Equilib} implies $I_1I_2=0$. By the made assumption this implies that exactly one of $I_1$ and $I_2$ is nonzero while another vanishes. This yields $G_3$ and $G_4$ in \eqref{subeqpt3} and \eqref{subeqpt4}, respectively. Now, let $I_{12}\ne0$ but $I_1I_2=0$. Then the last equation in \eqref{Equilib} implies $\alpha_3S+ \eta_1I_1+\eta_2I_2-\mu_3=0$. An elementary analysis reveals exactly three possible points $G_5, G_6$ and $G_7$ in \eqref{subeqpt5}--\eqref{subeqpt7}. Finally, consider the case when all coordinates of $Y$ are distinct from zero. Since $Y$ is distinct from $G_1$ and $G_2$, it must satisfy \eqref{law1}, \eqref{law2}. Also, since $I_1,I_2\ne0$,  we obtain from the second and the third equations \eqref{Equilib} the following system:
\begin{align*}
\mu_1I_1+\mu_2I_2+\mu_3I_{12}&=\frac{r}{K}(K-S)S,\\
\alpha_1I_1+\alpha_2I_2+\alpha_3I_{12}&=\frac{r}{K}(K-S),\\
\alpha_1S-\gamma_1I_2 - \eta_1I_{12} - \mu_1&=0,\\
\alpha_2S-\gamma_2I_1 - \eta_2I_{12}- \mu_2&=0.
\end{align*}
Rewriting  these four equations in the matrix form as follows
\begin{equation}\label{PSmat}
\left(
\begin{matrix}
\mu_1 & \mu_2 & \mu_3 & \frac{r}{K}(S-K)S \\
\alpha_1 & \alpha_2 & \alpha_3 &\frac{r}{K}(S-K) \\
0 & \gamma_1 & \eta_1 & \mu_1-\alpha_1S \\
\gamma_2 & 0 & \eta_2 & \mu_2-\alpha_2S \\
\end{matrix}
\right)
\left(
\begin{matrix}
I_1 \\
I_2 \\
I_{12}\\
1\\
\end{matrix}
\right)=
\left(
\begin{matrix}
0 \\
0 \\
0\\
0\\
\end{matrix}
\right)
\end{equation}
we conclude that $(I_1,I_2,I_{12},1)^T$ is a $0$-eigenvector of the matrix in the left hand side of \eqref{PSmat}, thus, the first coordinate $S$  satisfies the  determinant equation
$$
P(S):=p_2S^2+p_1S+p_0=0,
$$
where
\begin{equation}\label{Peq2}
P(S):=
\begin{vmatrix}
\mu_1 & \mu_2 & \mu_3 & \frac{r}{K}(S-K)S \\
\alpha_1 & \alpha_2 & \alpha_3 & \frac{r}{K}(S-K) \\
0 & \gamma_1 & \eta_1 & \mu_1-\alpha_1S \\
\gamma_2 & 0 & \eta_2 & \mu_2-\alpha_2S \\
\end{vmatrix}
\end{equation}
and
$$
p_0=\begin{vmatrix}
\mu_1 & \mu_2 & \mu_3 & 0 \\
\alpha_1 & \alpha_2 & \alpha_3 & \mu_0-b \\
0 & \gamma_1 & \eta_1 & \mu_1 \\
\gamma_2 & 0 & \eta_2 & \mu_2 \\
\end{vmatrix}
,\quad p_1=\begin{vmatrix}
\mu_1 & \mu_2 & \mu_3 & \mu_0-b \\
\alpha_1 & \alpha_2 & \alpha_3 &\frac{r}{K} \\
0 & \gamma_1 & \eta_1 & -\alpha_1 \\
\gamma_2 & 0 & \eta_2 & -\alpha_2 \\
\end{vmatrix}
,\quad p_2=\begin{vmatrix}
\mu_1 & \mu_2 & \mu_3 & \frac{r}{K} \\
\alpha_1 & \alpha_2 & \alpha_3 & 0 \\
0 & \gamma_1 & \eta_1 & 0 \\
\gamma_2 & 0 & \eta_2 & 0 \\
\end{vmatrix}
$$
In particular, it follows that $P(S)$ is a quadratic polynomial in $S$, therefore there may be at most two distinct inner points of type $G_8$. The condition $P(S)=0$ is sufficient if $\gamma_1,\gamma_2<\frac{\Delta_\alpha}{\alpha_3}$.
\qed
\end{proof}

It follows from Proposition~\ref{pro:equil} that all the boundary (edge) points are uniquely determined and can be expressed very explicitly. The existence and uniqueness of coexistence (inner) points of type $G_8$ is more involved (in contrast with the Lotka-Volterra case $\bar \gamma=0$) and depends on the value of $\bar \gamma$.

We study the existence and the local stability  of inner points by a bifurcation approach in the forthcoming paper \cite{forthcoming}.
Notice also that in the particular case $\gamma_i=0$, the characteristic polynomial \eqref{Peq2} becomes a linear function expressed explicitly by
\begin{align*}
P(S)|_{\gamma_1=\gamma_2=0}&=
%(\mu_1\alpha_2-\mu_2\alpha_1)(\eta_1\mu_2-\eta_2\mu_1 -S(\eta_1\alpha_2-\eta_2\alpha_1))
\alpha_1\alpha_2(\sigma_1-\sigma_2)(\Delta_\mu-S\Delta_\alpha)
%\\&=-\alpha_1\alpha_2(\sigma_2-\sigma_1) (\eta_1(\sigma_2-\sigma_1)\alpha_2-(S-\sigma_1)\Delta_\alpha),
\end{align*}
where we used the notation in \eqref{deltamu1}. This considerably simplifies the analysis, see \cite{SKTW18a}.

%Let $\Gamma$ be the boundary of $\Bbb R^4$, i.e.
%$$
%\Gamma=\{x=(x_1,x_2,x_3,x_4)\,:\,x_j\geq 0, j=1,\ldots,4,\;x_1x_2x_3x_4=0\}.
%$$

\begin{lemma}
The following holds:

\begin{itemize}
\item[(i)] For each $G_j$, $j=1,2,3,5$, there exists $\varepsilon>0$ (depending on the fundamental parameters $\alpha_i, \mu_i$, $\eta_i$ and $\gamma_i$) such that $\|G_j-G_8\|\geq\varepsilon$.

\item[(ii)] Let $G_4$ is given by \eqref{G4coor} and $\delta:=\alpha_1S^*-\gamma_1I_2^*-\mu_1> 0$ (or equivalently $\gamma^*<K/(K-\sigma_2)$). Then there exists $\varepsilon(\delta)>0$ such that $\|G_4-G_8\|\geq\varepsilon(\delta)$.

\item[(iii)] Let $G_6$ is given by \eqref{subeqpt6} and $\delta:=\alpha_2S^*-\eta_2I_{12}-\gamma_2I_1^*-\mu_2\neq 0$. Then there exists $\varepsilon(\delta)>0$ such that $\|G_6-G_8\|\geq\varepsilon(\delta)$.

\item[(iv)] Let $G_7$ is given by \eqref{subeqpt8} and $\delta:=\alpha_1S^*-\eta_1I_{12}-\gamma_1I_1^*-\mu_1\neq 0$. Then there exists $\varepsilon(\delta)>0$ such that $\|G_7-G_8\|\geq\varepsilon(\delta)$.
\end{itemize}
\end{lemma}

\begin{proof}

(i) We prove the assertion for $j=5$ since other cases are considered in a similar way.  The second and the third equations in \eqref{Equilib} near the point $G_5$ have the form
\begin{equation}\label{K1}
(\alpha_1K-\mu_1+O(\epsilon))I_1=0,\;\;(\alpha_2K-\mu_2+O(\epsilon))I_2=0,
\end{equation}
where $\epsilon=\|G_5-G_8\|$. By the assumption \eqref{assum}, one of numbers $\alpha_1K-\mu_1$, $\alpha_2K-\mu_2$ does not vanish and so the corresponding coefficient in (\ref{K1}) does not vanish also for small $\epsilon$, which implies (i) for $G_5$. Proofs of (ii)--(iv) use the same argument.

\end{proof}

\subsection{Equilibrium branches}
It turns out that the most natural way to study equilibrium points is to consider their dependence on the  carrying capacity $K $ (or, equivalently, on the modified  carrying capacity $K$). We know by Proposition~\ref{pro:G2} that the disease-free equilibrium point $G_2$ is the only stable equilibrium point for $0\le K< \sigma_1$. In this section we consider each equilibrium state separately and study their local stability for $K\ge \sigma_1$. We study first the local stability of each point individually and in the next sections consider the dependence on $K$.

Our main goal is to describe all possible continuous scenarios of how the locally stable equilibrium states of \eqref{MAIN} depends on $K$ provided that all other fundamental parameters $\alpha_i$, $\mu_i$, $b$, $\gamma_i$ remain fixed. To this end, we introduce the following concept.

\begin{definition}
By an \textit{equilibrium branch} we understand any continuous in $K\ge0$ family of equilibrium points of \eqref{MAIN} which are locally stable for all but finitely many threshold values of $K$.
\end{definition}

\begin{remark}
We need to distinguish the threshold values of $K$ in the above definition because, formally, the local stability (i.e. that the real parts of all the system characteristic roots are negative) fails when an equilibrium point change its type. On the other hand, a branch may be stable in the Lyapunoff sense even for the threshold values of $K$. Indeed, the latter holds at least for $\gamma=0$, see \cite{SKTW18a}.
\end{remark}

\section{The equilibrium state $G_3$: Proof of \ref{it1}}
%\subsection{The equilibrium state $G_3$: Proof of \ref{it1}}
Note that the next three boundary equilibriums $G_3,G_4$ and $G_5$ have  \textit{a priori}  non-zero coordinates; furthermore, the $S$-coordinate is a constant (independent on $K$).
The first of these is the equilibrium point $G_3$ with the presence of only the first strain. Its explicit expression with the nonnegativity condition are given by \eqref{subeqpt3}. Remark that when $K= \sigma_1$, the globally stable point $G_2$ bifurcates into $G_3$:
$$
\text{ $G_3 = G_2$ \quad when \quad $I_1^* = 0$ $\Leftrightarrow$ $K= \sigma_1$}
$$
Using \eqref{subeqpt3}, we find the corresponding Jacobian matrix evaluated at $G_3=(\sigma_1, I_1^*, 0,0)$:
%where
%\begin{equation}\label{eqpoint3}
%\begin{split}
%S^* &=\frac{\mu_1}{\alpha_1},\\
%Y_1^* &=\frac{b}{\Omega \alpha_1}(K-\frac{\mu_1}{\alpha_1}),
%\end{split}
%\end{equation}
\[
J_3=\begin{bmatrix}
%A & *  \\
%0& B
%\end{bmatrix}
%=
%\begin{bmatrix}
-\frac{r\sigma_1}{K} & -\alpha_1\sigma_1 & -\alpha_2\sigma_1 & -\alpha_3\sigma_1  \\
\alpha_1I_1^* & 0 & -\gamma_1I_1^* & -\eta_1I_1^* \\
0 & 0 & -\alpha_2(\sigma_2-\sigma_1)-\gamma_2I_1^* & 0 \\
0 & 0 & \overline{\gamma}I_1^* & -\alpha_3(\sigma_3-\sigma_1)+\eta_1I_1^*
\end{bmatrix},
\]
Notice that, $J_3$ has a block structure. The left upper $2\times 2$-block is obviously stable. Therefore $J_3$ is stable if and only if the right lower block is so. By virtue of $-\alpha_2(\sigma_2-\sigma_1)-\gamma_2I_1^*<0$ this is equivalent to
\begin{equation}\label{detG3}
-\alpha_3(\sigma_3-\sigma_1)+\eta_1I_1^* <0,
\end{equation}
or, equivalently, using the expression $I_1^* =\frac{r}{K \alpha_1}(K-\sigma_1)$ and \eqref{A1} we obtain
\begin{equation}\label{Ethreshold1}
\eta_1^*< \frac{K}{K-\sigma_1}.
\end{equation}
After some obvious manipulations we arrive at

\begin{proposition}\label{pro:G3}
The equilibrium point $G_3$ is stable nonnegative if and only if either \begin{equation}\label{G3ineq}
\left\{
\begin{array}{cc}
K>\sigma_1&\quad\text{if $\eta_1^*\le 1$}\\
&\\
\sigma_1<K<K_1& \quad\text{if $\eta^*_1>1$}.
\end{array}
\right.
\end{equation}
\end{proposition}

Notice that the point $G_3$ remains nonnegative and locally stable for \textit{any}  $K>\sigma_1$ provided $\eta^*_1\le 1$. This provides us with the first (simplest) example of a branch. More precisely, we have

\begin{corollary}[Branch \ref{it1}]\label{cor:G3}
Let $\eta^*_1\le 1$. Then
\begin{enumerate}[label=(\alph*),itemsep=0.4ex,leftmargin=0.8cm]
  \item for $0< K<\sigma_1$ the point $G_2$ is locally (in fact, globally) stable;
  \item for $ K=\sigma_1$ the point $G_2$ coincides with $G_3$;
  \item for $K>\sigma_1$ the point $G_3$ is locally stable.
\end{enumerate}
We display this schematically as
$$
G_2\rightarrow G_3
$$
\end{corollary}

The latter corollary implies \ref{it1} in Theorem~\ref{theo1}.

\section{Proof of \ref{it4}}
Corollary~\ref{cor:G3} completely describes all possible scenarios for $0\le K<\infty$ when $\eta^*_1\le1$. In what follows, we shall always assume that $\eta^*_1>1$. Then Proposition~\ref{pro:G3} tells us that $G_3$ remains locally stable for any $\sigma_1<K<K_1$. If we want to find a \textit{continuous} equilibrium branch, we need to check which of the remained candidates $G_4,G_5,G_6,G_7,G_8$ becomes equal to $G_3$ for the right critical value $K=K_1$.

An easy inspection shows that for a generic choice of the fundamental parameters there is only one possible candidate, namely $G_6$. Thus, to construct the only possible scenario for a continuous equilibrium branch is when $G_3$ bifurcates into $G_6$. In the next section we give stability analysis of $G_4$ and $G_5$, and then continue with $G_6$ and construction of equilibrium branches.
%This also shows that the branch \eqref{branch236} is generically unique.

%Therefore, det $B$ > 0 implies that tr $B = B_{11} + B_{22} < 0$, i.e. �� is stable. This shows that the block $B$ is stable if and only if det $B> 0$ holds.
%where $I_1$ is given by \eqref{subeqpt3} and
%has one negative eigenvalue, i.e
%$$\lambda_3=
%-\alpha_2(\sigma_2-\sigma_1)-\gamma_2I_1.$$
%The other eigenvalue
%$$\lambda_4=-\alpha_3(\sigma_3-\sigma_1)+\eta_1I_1 $$
%is negative if
%$\alpha_3(\frac{\mu_1} {\alpha_1}-\frac{\mu_3} {\alpha_3})$
%$\eta_1 <\hat{\eta_1} %\frac{\Omega \alpha_1\alpha_3}{b(K-\sigma_1)}(\sigma_3-\sigma_1),
%$
%and then the remaining $2\times2$ block has negative trace and positive %determinant as follows
%$$\tau=-\frac{b}{\Omega }\sigma_1$$
%$$\Delta=\alpha_1^2\sigma_1I_1.$$
 %Therefore the $G_3$ is stable.

\subsection{The equilibrium state $G_4$}
The equilibrium point $G_4$ expresses the presence of only the second strain, see \eqref{G4coor}.  It
is nonnegative if and only if
\begin{equation}\label{subeqpt4}
K>\sigma_2 .
\end{equation}
%where
%\begin{equation}\label{eqpoint4}
%\begin{split}
%S^* &=\frac{\mu_2}{\alpha_2},\\
%Y_2^* &=\frac{b}{\Omega \alpha_2}(K-\frac{\mu_2}{\alpha_2}),
%\end{split}
%\end{equation}
% and $G_4$ is nonnegative if and only if
 %$$I_2^*=\frac{b}{\Omega \alpha_2}(K-\sigma_2)>0$$
 Note that if $G_4$ is nonnegative then  by virtue of \eqref{subeqpt4} and \eqref{assum}, $G_3$ is nonnegative too.
The Jacobian matrix evaluated at $G_4$ is
\begin{equation}\label{J.G4}
J_4=
\begin{bmatrix}
-r\frac{\sigma_2}{K} & -\alpha_1\sigma_2 & -\alpha_2\sigma_2 & -\alpha_3\sigma_2  \\
0 & \alpha_1(\sigma_2-\sigma_1)-\gamma_1I_2^* & 0 & 0 \\
\alpha_2I_2^* & -\gamma_2I_2^* & 0 & -\eta_2I_2^* \\
0 & \overline{\gamma}I_2^* & 0 & -\alpha_3(\sigma_3-\sigma_2)+\eta_2I_2^*
\end{bmatrix}
\end{equation}
Note that, interchanging rows and columns of the matrix \eqref{J.G4} only change the sign of the determinant of this matrix. Therefore, after an obvious rearrangement, the eigenvalues of $J_4$ solves the following equation:
\begin{equation}\label{J.G4 diagonalilized}
\begin{vmatrix}
-r\frac{\sigma_2}{K}-\lambda &-\alpha_2\sigma_2 & -\alpha_1\sigma_2 & -\alpha_3\sigma_2  \\
\alpha_2I_2^* & -\lambda& -\gamma_2I_2^* & -\eta_2I_2^* \\
0 & 0 & \alpha_1(\sigma_2-\sigma_1)-\gamma_1I_2^* -\lambda& 0 \\
0 & 0& \overline{\gamma}I_2^* & -\alpha_3(\sigma_3-\sigma_2)+\eta_2I_2^*-\lambda
\end{vmatrix}=0.
\end{equation}
Again, one easily verifies that the left upper $2\times2$-block is stable, while the stability of the right down (lower-diagonal) block is equivalent to the negativity of the diagonal elements, i.e. to  the inequalities
$$
\left\{
\begin{array}{rl}
\alpha_1(\sigma_2-\sigma_1)-\gamma_1I_2^* &<0,\\
 -\alpha_3(\sigma_3-\sigma_2)+\eta_2I_2^*&<0.
\end{array}
\right.
$$
 Thus the stability of $G_4$ is equivalent to the inequalities
 \begin{equation}\label{G4ineq}
\left\{
\begin{array}{rl}
 K(1-\frac{1}{\gamma^*})>\sigma_2\\
 %\sigma_2<K<\frac{\sigma_1}{1-\frac{A_2}{\eta_2}}
 K<K_3,
\end{array}
\right.
 \end{equation}
 where $\gamma^*:=\frac{\gamma_1}{A_3}$.
In summary, we have
\begin{proposition}
\label{pro:G4}
The equilibrium point $G_4$ is stable and nonnegative iff
\begin{itemize}
\item[$\bullet$]
$K_3 <K< \frac{\sigma_2\gamma^*}{\gamma^*-1}$ when $\gamma^*>1$ and $\eta_2^*>1$, or
\item[$\bullet$]
$K>\frac{\sigma_2\gamma^*}{\gamma^*-1}$ when $\gamma^*>1$ and $\eta_2^*<1$.
           \end{itemize}
 \end{proposition}

 \begin{remark}
 In this paper, we are primarily interested in the case of `small' values of $\gamma_i$. On the other hand, the latter proposition shows that $G_4$ may be stable only if $\gamma_1>A_3$, therefore this equilibrium is not stable for small values of $\gamma_1$ and will be eliminated from the subsequent analysis.
 \end{remark}

 \begin{corollary}
The equilibrium point $G_4$ is locally unstable if $0\le \gamma_1^*<1$.
 \end{corollary}

\subsection{The equilibrium state $G_5$}
An equilibrium point in the presence of coinfection is given by \eqref{subeqpt5}.

\begin{proposition}
\label{pro:G5}
The equilibrium point $G_5$ is stable and nonnegative iff
\begin{equation}\label{G5e}
%\max\{\frac{1}{\eta_1^*}, \frac{A_2}{\eta_2}\}<1, \quad
\eta:=\min\{\eta_1^*, \eta_2^*\}>1 \quad \text{and}\quad K>\frac{\sigma_3\eta}{\eta-1}.
 \end{equation}
 Furthermore,  if the point $G_5$ is nonnegative and locally stable for a certain $K_0>0$ then it will be so for any $K\ge K_0$ (provided that other parameters are fixed).
\end{proposition}

\begin{proof}
By \eqref{subeqpt5}, $I_{12}^*=\frac{r}{K \alpha_3}(K-\sigma_3)$, hence the positivity of $I_{12}^*$ is equivalent to
$$
K> \sigma_3.
$$
Next, the Jacobian matrix evaluated at $G_5$ is
\begin{equation}\label{J.G5}
J_5=
\begin{bmatrix}
-r\frac{\sigma_3}{K} & -\alpha_1\sigma_3 & -\alpha_2\sigma_3 & -\alpha_3\sigma_3  \\
0 & \alpha_1(\sigma_3-\sigma_1)-\eta_1I_{12}^* & 0 & 0 \\
0 & 0 & \alpha_2(\sigma_3-\sigma_2)-\eta_2I_{12}^* & 0 \\
\alpha_3I_{12}^* & \eta_1I_{12}^* & \eta_2I_{12}^* & 0
\end{bmatrix},
\end{equation}
The matrix has a block structure where the block
$$
\begin{bmatrix}
-r\frac{\sigma_3}{K} & -\alpha_3\sigma_3  \\
\alpha_3I_{12}^* & 0
\end{bmatrix}
$$
is obviously stable, therefore the stability of $J_5$ is equivalent to the negativity of two diagonal elements:
\begin{align*}
\alpha_1(\sigma_3-\sigma_1)-\eta_1I_{12}^*<0,\\
\alpha_2(\sigma_3-\sigma_2)-\eta_2I_{12}^*<0.
\end{align*}
First notice that stability of $G_5$ implies immediately that $I^*_{12}>0$. Also, taking into account that $I_{12}^*=\frac{r}{K \alpha_3}(K-\sigma_3)$,  the stability of $G_5$ is equivalent to the inequalities
 %\begin{equation}\label{G5ineq}
%\begin{split}
%K>\frac{\sigma_1}{1-\frac{1}{\eta_1^*}}\\
%K>\frac{\sigma_1}{1-\frac{A_2}{\eta_2}}
%\end{split}
%\end{equation}
$$
\sigma_3<  K\left(1-\min\left\{\frac{1}{\eta^*_1},\,    \frac{1}{\eta^*_2}\right\}\right)=K(1-\frac{1}{\eta}).
$$
In summary, we have \eqref{subeqpt5}. Finally, the last statement of the proposition follows immediately from the increasing (with respect to $K$) character  of the second inequality in \eqref{G5e}.
\end{proof}

%\begin{remark}
%Note that it follows a priori from \eqref{G5e} that $K>\sigma_3$ when $G_5$ is stable. Also, using \eqref{submodel2}, the boundary equilibrium point $G_{1+j}$ for $j=1,2,3$ has non-negative coordinates  when $K > \sigma_j$ and it coincides with $G_2$ if $K=\sigma_j$.
%\end{remark}

\begin{remark}
We emphasize that the stability of the equilibrium states $G_2,G_3,G_4$ and $G_5$ does not involve the interference constants $\gamma_1, \gamma_2$.
\end{remark}

\subsection{The equilibrium state $G_6$}\label{sec:g6}
Analysis of the remaining three equilibrium points $G_6,G_7$ and $G_8$ is more delicate and now also involves the coinfection constants $\gamma_1,\gamma_2.$
Let us  consider the  boundary equilibrium point  $$
G_6=(\frac{\sigma_1K}{K_1}, \,\,
\frac{\mu_3}{\eta_1}(1-\frac{K}{K_2}),\,\,0,\,\,
\frac{\mu_1}{\eta_1}(\frac{K}{K_1}-1)),
$$
see \eqref{subeqpt6}. First notice that the coordinates of $G_6$ are nonnegative if and only if the two conditions hold: $K_1>0$, what is equivalent to $\eta^*_1>1$, and also
$$
\sigma_1<S^*<\sigma_3.
$$
We see that $G_6$ is nonnegative if and only if
\begin{equation}\label{G6pos}
K_1<K<K_2, \qquad \eta^*_1>1.
\end{equation}
(Note that the bilateral inequality is inconsistent with \eqref{G5e}).

Now let us study the local stability of $G_6$. Using \eqref{subeqpt6}, the Jacobian matrix for $G_6$ is found as
$$
J_6=
\begin{bmatrix}
-r\frac{S^*}{K } & -\alpha_1S^* & -\alpha_2S ^*& -\alpha_3S^^*  \\
\alpha_1I_1^* & 0 & -\gamma_1I_1^* & -\eta_1I_1^* \\
0 & 0 & \alpha_2S^*-\eta_2I_{12}^*-\gamma_2I_1^*-\mu_2 & 0 \\
\alpha_3I_{12}^* & \eta_1I_{12}^* & \eta_2I_{12}^*+\overline{\gamma}I_1^* & 0
\end{bmatrix}.
$$
with $S^*,I^*_1,I^*_{12}$  given by \eqref{subeqpt6}. Using the block structure of $J_6$, it follows that $G_6$ is locally stable if and only if
\begin{itemize}
\item[$\bullet$] there holds
\begin{equation}\label{stabcond6.1}
 \alpha_2S^*-\eta_2I_{12}^*-\gamma_2I_1^*-\mu_2 <0
\end{equation}
\item[$\bullet$] and the matrix below is stable:
\begin{equation}\label{JII}
\tilde J=
\begin{bmatrix}
-r\frac{S^*}{K} & -\alpha_1S ^*& -\alpha_3S ^* \\
\alpha_1I_1 ^*& 0 & -\eta_1I_1^* \\
\alpha_3I_{12}^* & \eta_1I_{12}^* & 0
\end{bmatrix}
=
\begin{bmatrix}
S^* & 0 & 0 \\
0 & I_1^* & 0 \\
0 & 0 & I_{12}^*
\end{bmatrix}
\begin{bmatrix}
-\frac{r}{K} & -\alpha_1 & -\alpha_3  \\
\alpha_1 & 0 & -\eta_1 \\
\alpha_3 & \eta_1 & 0
\end{bmatrix}.
\end{equation}
\end{itemize}

The stability of $\tilde J$ is equivalent to the stability of the last matrix factor in \eqref{JII}. An easy application of the Routh-Hurwitz criteria \cite{Gantmacher} confirms that $\tilde J$ is always stable. Hence, the stability of $G_6$ is equivalent to  the condition \eqref{stabcond6.1}. Using \eqref{subeqpt6}, we can rewrite it as follows:
\begin{equation}\label{stabcond6.2}
    S^*(\Delta_\alpha+\gamma_2\alpha_3)<\Delta_\mu+\gamma_2\mu_3
\end{equation}
see \eqref{deltamu}. Let us define
\begin{equation}\label{S1}
\hat{S}_1:=\frac{\Delta_\mu+\mu_3\gamma_2}{\Delta_\alpha+\alpha_3\gamma_2}
\end{equation}
We have by using \eqref{A1}--\eqref{A3}
\begin{align}\label{Shat1}
\hat{S}_1-\sigma_1&=
\frac{\eta_1\alpha_1\alpha_2(\sigma_2-\sigma_2)+ \gamma_2\alpha_1\alpha_3(\sigma_3-\sigma_1)}{\alpha_1(\Delta_\alpha+\alpha_3\gamma_2)}\nonumber\\
&=\frac{r(\eta_1A_3+\gamma_2A_1)}{\alpha_1(\Delta_\alpha+\alpha_3\gamma_2)},\\
\label{Shat1sigma2}
\hat{S}_1-\sigma_2
&=\frac{\eta_2\alpha_1\alpha_2(\sigma_2-\sigma_1) +\gamma_2\alpha_2\alpha_3(\sigma_3-\sigma_2)}{\alpha_2 (\Delta_\alpha+\alpha_3\gamma_2)}\nonumber\\
&=\frac{r(\eta_2A_3+\gamma_2A_2)}{\alpha_1(\Delta_\alpha+\alpha_3\gamma_2)},\\
\label{Shat1sigma3}
\hat{S}_1-\sigma_3
&=\frac{\eta_2\alpha_1\alpha_3(\sigma_3-\sigma_1)-\eta_1\alpha_2\alpha_3 (\sigma_2-\sigma_3)}{\alpha_3(\Delta_\alpha+\alpha_3\gamma_2)}\nonumber\\
&=\frac{A_1A_2r(\eta_2^*- \eta_1^*)}{\alpha_3(\Delta_\alpha+\alpha_3\gamma_2)},
\end{align}

Consider first the case $\Delta_\alpha+\alpha_3\gamma_2=0$. Then by \eqref{deltamu21} it follows that $\Delta_\mu+\gamma_2\mu_3>0$ therefore \eqref{stabcond6.2} holds automatically true in this case, and $G_6$ is locally stable.

Next consider the case $\Delta_\alpha+\alpha_3\gamma_2<0$. Then it follows from \eqref{stabcond6.2} that $G_6$ is stable whenever $S^*>\hat{S}_1$. On the other hand, \eqref{Shat1} implies in this case $\hat{S}_1<\sigma_1,$ therefore using \eqref{zineq} we see that
\begin{equation}\label{Shatsigma2}
S^*>\sigma_1>\hat{S}_1
\end{equation}
whenever $S^*$ is nonnegative. Therefore in this case $G_6$ is locally stable whenever \eqref{G6pos} are fulfilled. Note also that under the made assumption $\Delta_\alpha+\alpha_3\gamma_2<0$ one necessarily has $\eta_2A_1> \eta_1 A_2$. Indeed, if $\eta_2A_1\le \eta_1 A_2$ then \eqref{etaA13} implies $\Delta_\alpha>0$, therefore $\Delta_\alpha+\alpha_3\gamma_2>0$, a contradiction.

Finally, assume that
\begin{equation}\label{Deltaalpha1}
\Delta_\alpha+\alpha_3\gamma_2>0
\end{equation}
Then  by \eqref{stabcond6.2} the point $G_6$ is locally stable if and only if $S^*<\hat{S}_1$, i.e.
\begin{equation}\label{S1hatstar}
K<\frac{\hat{S}_1}{1-\frac{1}{\eta_1^*}}.
\end{equation}
Under assumption \eqref{Deltaalpha1}, \eqref{Shat1} implies $\hat{S}_1>\sigma_1$. On the other hand, we have
$$
\hat{S}_1\ge\sigma_3
\text{ if }
\eta_2^*\ge\eta^*_1
\quad\text{ and }\quad
\hat{S}_1<\sigma_3
\text{ if }
\eta_2^*<\eta^*_1.
$$
On the other hand, in the latter case, the inequality $\eta_2^*\ge\eta^*_1$  by virtue of \eqref{etaA13} that in fact
$\Delta_\alpha>0$, therefore \eqref{Deltaalpha1} holds automatically true in this case. Combining \eqref{S1hatstar} with the nonnegativity condition \eqref{G6pos}, and summarizing the above observations we arrive at
%In summary, this implies

\begin{proposition}
\label{pro:G6}
The equilibrium point $G_6$ is nonnegative stable  iff $\eta_1>A_1$ and the following conditions hold:
\begin{equation}\label{conG6}
K_1<K<\frac{Q}{\sigma_1}K_1
\end{equation}
where
\begin{equation}\label{Qdef}
Q=\left\{
\begin{array}{ll}
\sigma_3&\quad \text{if  $\eta_2^*\ge\eta^*_1$};\\
\hat{S}_1&\quad \text{if $\eta_2^*<\eta^*_1$.}
\end{array}
\right.
\end{equation}
\end{proposition}

%\begin{remark}
%We notice that  the inequalities $\eta_2A_1<\eta_1 A_2$ imply by virtue of \eqref{etaA13} that in fact
%$$
%\Delta_\alpha>0.
%$$
%In other words, \eqref{conG6} with the nontrivial upper bound $\frac{\hat{S}_1}{1-\frac{1}{\eta_1^*}}$ holds only if $\Delta_\alpha>0$. This remark will be useful later on when we discuss bifurcations of $G_6$.
%\end{remark}

Now we are ready to describe the equilibrium branch for $\eta_1^*>1$.

\begin{corollary}\label{cor:G6}
Let $\eta_1^*>1$. Then
\begin{enumerate}[label=(\alph*),itemsep=0.4ex,leftmargin=0.8cm]
\item for $0< K<\sigma_1$ the point $G_2$ is locally (in fact, globally) stable;
\item for $K=\sigma_1$ the point $G_2$ coincides with $G_3$;
\item for $\sigma_1<K<K_1$ the point $G_3$ is locally stable;
\item\label{G6-4} for $K=K_1$ the
      point $G_3$ coincides with $G_6$;
\item\label{G6-5} for $K_1<K<\frac{Q}{\sigma_1}K_1$ the point $G_6$ is locally stable, where $Q$ is defined  by \eqref{Qdef}.
\end{enumerate}
We display this schematically as
\begin{equation}\label{branch236}
G_2\rightarrow G_3\rightarrow G_6\rightarrow \ldots
\end{equation}
\end{corollary}

\begin{proof}
The first three items are obtained by combining Proposition ~\ref{pro:G3} with Proposition~\ref{pro:G2}. Note that the upper bound in (c) here is smaller than that in (c) in Corollary~\ref{cor:G3}. When $K=K_1=\frac{\sigma_1\eta_1^*}{\eta_1^*-1}$, it follows that the $I_{12}$-coordinate of $G_6$ vanishes (see \eqref{subeqpt6}), i.e. $G_6=G_3$, which proves \ref{G6-4}. Next,  Proposition~\ref{pro:G6} yields \ref{G6-5}.
\end{proof}

With Corollary~\ref{cor:G3} and Corollary~\ref{cor:G6} in hand, it is natural to ask:
What happens with an equilibrium branch  when $\eta^*_1>1$ and $K>K_1$?

So far, we see that any continuous equilibrium branch develops uniquely determined accordingly \eqref{branch236}. But at $G_6$ the situation becomes more complicated: this point may a priori bifurcate  into different points.

In this paper we only consider the particular case \ref{it4}, i.e. when $1<\eta_1^*<\eta_2^*$. This yields by \eqref{Qdef} that $Q=\sigma_3$, hence \eqref{conG6} implies that $G_6$ is locally stable for
$$
K_1<K<K_2.
$$
The upper critical value $K_2$ substituted  in \eqref{subeqpt6} implies that $I_1^*=0$, hence $G_6$ naturally bifurcates into $G_5$. It is easy to see that the corresponding $I_{12}^*$ for $G_5$ and $G_6$ coincide when $K=K_2$ holds. This observation combined with Proposition~\ref{pro:G5} implies that in this case for any $K>K_2$ the point $G_5$ will be locally stable, hence we arrive at

\begin{corollary}[Branch \ref{it4}]
\label{cor:G6-G5}
Let $\eta_2^*\ge \eta_1^*>1$ hold. Then
\begin{enumerate}[label=(\alph*),itemsep=0.4ex,leftmargin=0.8cm]
\item for $0< K<\sigma_1$ the point $G_2$ is locally (in fact, globally) stable;
\item for $K=\sigma_1$ the point $G_2$ coincides with $G_3$;
\item for $\sigma_1<K<K_1$ the point $G_3$ is locally stable;
\item\label{G65-4} for $K=K_1$ the
      point $G_3$ coincides with $G_6$;
\item\label{G65-5} for $K_1<K<K_2$ the point $G_6$ is locally stable;
\item\label{G65-6} for $K=K_2$ the point $G_6$ coincides with $G_5$;
\item\label{G65-7} for $K>K_2$ the point  $G_5$ is locally stable.
\end{enumerate}
We display this schematically as
\begin{equation}\label{branch2365}
G_2\rightarrow G_3\rightarrow G_6\rightarrow G_5
\end{equation}
\end{corollary}

\subsection{Bifurcation of $G_6$}\label{sec:pr}

Thus, we are remained to study the  case when
\begin{equation}\label{etaA1}
\eta_2^*< \eta_1^*, \qquad \eta_1^*>1
\end{equation}
hold. Notice that in fact by virtue of \eqref{etaA13} the latter inequality implies
\begin{equation}\label{Deltalphaplus}
\Delta_\alpha>0.
\end{equation}
We know by \ref{G6-5} in Corollary~\ref{cor:G6} that $G_6$ is locally stable for
$$
K_1<K<\frac{\hat S_1\eta_1^*}{\eta_1^*-1}.
$$
Substituting the corresponding critical value $K =K_0$ such that
$$
K_0=\frac{\hat S_1\eta_1^*}{\eta_1^*-1}=
\frac{\Delta_\mu+\mu_3\gamma_2}{\Delta_\alpha+\alpha_3\gamma_2}\cdot
\frac{\eta_1^*}{\eta_1^*-1}
$$
in \eqref{subeqpt6} reveals that the coordinates $G_6$ \textit{do not vanish}, i.e. $G_6$ does not change its type. Instead it losts its local stability because the determinant of $J_6$ vanishes at this moment. To continue the equilibrium branch \eqref{branch236} beyond $G_6$ we need to find an appropriate candidate for a stable point. By the continuity argument (because $G_6$ keeps all coordinates nonzero for $K=K_0$), the only possible candidate for a continuous equilibrium branch is a point of type $G_8$. Since we do not have any explicit expression of $G_8$, the analysis in this case is more complicated and involves a certain bifurcation technique which we develop in a forthcoming paper \cite{forthcoming}.

\section{Concluding remarks}
It is  natural, from biological point of view, to relax the constancy condition on the transmission rates $\alpha_i$ and assume that in general they may depend on the carrying capacity. Indeed, a larger carrying capacity corresponds to a larger size of the population (or  the susceptible class), thus it must imply a slower spread of strains, i.e. $\alpha_i=\alpha_i(K)$ must be a non-increasing function. One natural assumptions is the following relation:
\begin{equation}\label{az}
\alpha_i(K)=\frac{a_i}{K}.
\end{equation}
This implies for the other fundamental constants
$$
\sigma_i=\frac{\mu_i}{a_i}K=:s_iK,
$$
and
$$
A_i=\frac{B_i}{K}, \quad \text{where }\quad B_1=\frac{a_1a_3(s_3-s_1)}{r} \text{ etc.}
$$
The main consequence of \eqref{az} is that the coordinates of a stable equilibrium point is no longer bounded and develop as $K$ increases.
For example,  under assumption \eqref{az} one has from \eqref{zineq} merely
$$
s_1K\le S\le K\min\{1,s_3\}.
$$
This, in particular implies that already the first bifurcation $S_2\to S_3$ is completely different. Indeed, it follows from Proposition~\ref{pro:G2} that $G_2$ becomes stable \textit{for all }$K>0$ provided $s_1\ge1$. In the nontrivial case $s_1<1$,  $G_2$ is \textit{never} stable. In general,  Proposition~\ref{pro:G3} and Corollary~\ref{cor:G6-G5}  instead imply

%Furthermore, in the remained case $s_1<1$ and $\eta_1^*>\frac{1}{1-s_1}$ Corollary~\ref{cor:G6-G5} implies that the alternatives (a)--(c) are not relevant, but

\begin{corollary}\label{cornew}
We have the following stability analysis:
\begin{itemize}
\item[(i)] If $s_1\ge1$ then $G_2$ is stable for all $K>0$;
\item[(ii)] If $s_1<1$ and $0<\eta_1^*\le \frac{1}{1-s_1}$  then $G_3$ stable for all $K>0$;
\item[] Let now $s_1<1$, $\eta_2^*>\eta_1^*> \frac{1}{1-s_1}$ hold. Then
\item[(iii)]   if  $s_3\ge 1$ or $s_3<1$ and $\eta_1^*<\frac{1}{1-s_3}$ then $G_6$ stable for all $K>0$;
\item[(iv)]  if $s_3<1$ and $\eta_1^*>\frac{1}{1-s_3}$ then $G_5$ stable for all $K>0$.
\end{itemize}
\end{corollary}

Thus, we have a complete description in the cases $\eta_1^*\le 1$ and  $\eta_2^*\ge \eta_1^*>1$. The remained case $\eta_1^*\ge \max\{1,\eta_2^*\}$ will be considered in \cite{forthcoming}.

\bigskip
\noindent {\bf Acknowledgements.} Vladimir Kozlov was supported by the Swedish Research Council (VR), 2017-03837.
% Authors must disclose all relationships or interests that
% could have direct or potential influence or impart bias on
% the work:

 \section*{Data availability statement}
The manuscript has no associated data.

 \section*{Compliance with ethical standards}

\textbf{Conflict of interest:} The authors declare that they have no conflict of interests.

% BibTeX users please use one of
%\bibliographystyle{spbasic}      % basic style, author-year citations
%\bibliographystyle{spmpsci}      % mathematics and physical sciences
%\bibliographystyle{spphys}       % APS-like style for physics
%\bibliography{}   % name your BibTeX data base

\bibliographystyle{spmpsci}%
%\bibliography{references7}

\begin{thebibliography}{10}
\providecommand{\url}[1]{{#1}}
\providecommand{\urlprefix}{URL }
\expandafter\ifx\csname urlstyle\endcsname\relax
  \providecommand{\doi}[1]{DOI~\discretionary{}{}{}#1}\else
  \providecommand{\doi}{DOI~\discretionary{}{}{}\begingroup
  \urlstyle{rm}\Url}\fi

\bibitem{Allen2}
Ackleh, A.S., Allen, L.J.: Competitive exclusion and coexistence for pathogens
  in an epidemic model with variable population size.
\newblock Journal of mathematical biology \textbf{47}(2), 153--168 (2003)

\bibitem{Allen}
Allen, L.J., Langlais, M., Phillips, C.J.: The dynamics of two viral infections
  in a single host population with applications to hantavirus.
\newblock Math. Biosci. \textbf{186}(2), 191--217 (2003)

\bibitem{forthcoming}
Andersson, J., Ghersheen, S., Kozlov, V., Tkachev, V., Wennergren, U.: \ttitle,
  the bifurcation analysis  (2020).
\newblock Submitted

\bibitem{Bichara}
Bichara, D., Iggidr, A., Sallet, G.: Global analysis of multi-strains sis, sir
  and msir epidemic models.
\newblock Journal of Applied Mathematics and Computing \textbf{44}, 273--292
  (2014)

\bibitem{Boldin}
Boldin, B.: Introducing a population into a steady community: the critical
  case, the center manifold, and the direction of bifurcation.
\newblock SIAM J. Appl. Math. \textbf{66}(4), 1424--1453 (2006).
\newblock \doi{10.1137/050629082}.
\newblock \urlprefix\url{https://doi.org/10.1137/050629082}

\bibitem{Bremermann}
Bremermann, H.J., Thieme, H.: A competitive exclusion principle for pathogen
  virulence.
\newblock Journal of mathematical biology \textbf{27}(2), 179--190 (1989)

\bibitem{Castillo}
Castillo-Chavez, C., Velasco-Hernandez, J.X.: On the relationship between
  evolution of virulence and host demography.
\newblock Journal of theoretical biology \textbf{192}(4), 437--444 (1998)

\bibitem{CrandallRabinowitz1}
Crandall, M.G., Rabinowitz, P.H.: The principle of exchange of stability.
\newblock In: Dynamical systems ({P}roc. {I}nternat. {S}ympos., {U}niv.
  {F}lorida, {G}ainesville, {F}la., 1976), pp. 27--41 (1977)

\bibitem{DiekmannGetto}
Diekmann, O., Getto, P., Gyllenberg, M.: Stability and bifurcation analysis of
  {V}olterra functional equations in the light of suns and stars.
\newblock SIAM J. Math. Anal. \textbf{39}(4), 1023--1069 (2007/08).
\newblock \doi{10.1137/060659211}.
\newblock \urlprefix\url{https://doi.org/10.1137/060659211}

\bibitem{Diekmann1990}
Diekmann, O., Heesterbeek, J.A.P., Metz, J.A.J.: On the definition and the
  computation of the basic reproduction ratio {$R_0$} in models for infectious
  diseases in heterogeneous populations.
\newblock J. Math. Biol. \textbf{28}(4), 365--382 (1990).
\newblock \doi{10.1007/BF00178324}.
\newblock \urlprefix\url{https://doi.org/10.1007/BF00178324}

\bibitem{Gantmacher}
Gantmacher, F.R.: The theory of matrices. {V}ols. 1, 2.
\newblock Translated by K. A. Hirsch. Chelsea Publishing Co., New York (1959)

\bibitem{SKTW18a}
Ghersheen, S., Kozlov, V., Tkachev, V.G., Wennergren, U.: Dynamical behaviour
  of sir model with coinfection: the case of finite carrying capacity.
\newblock Math. Meth. Appl. Sci. \textbf{42}(8) (2019)

\bibitem{Gog}
Gog, J.R., Grenfell, B.T.: Dynamics and selection of many-strain pathogens.
\newblock Proceedings of the National Academy of Sciences \textbf{99}(26),
  17209--17214 (2002).
\newblock \doi{10.1073/pnas.252512799}.
\newblock \urlprefix\url{https://www.pnas.org/content/99/26/17209}

\bibitem{new_cite}
Marie, I.E., Masaomi, K.: Effects of metapopulation mobility and climate change
  in si-sir model for malaria disease.
\newblock In: Proceedings of the 12th International Conference on Computer
  Modeling and Simulation, ICCMS '20, p. 99–103. Association for Computing
  Machinery, New York, NY, USA (2020).
\newblock \doi{10.1145/3408066.3408084}.
\newblock \urlprefix\url{https://doi.org/10.1145/3408066.3408084}

\bibitem{Martcheva}
Martcheva, M., Pilyugin, S.S.: The role of coinfection in multidisease
  dynamics.
\newblock SIAM J. Appl. Math. \textbf{66}(3), 843--872 (2006).
\newblock \doi{10.1137/040619272}.
\newblock \urlprefix\url{https://doi.org/10.1137/040619272}

\bibitem{May4}
May, R.M., Nowak, M.A.: Coinfection and the evolution of parasite virulence.
\newblock Proceedings of the Royal Society of London. Series B: Biological
  Sciences \textbf{261}(1361), 209--215 (1995)

\bibitem{Mosquera}
Mosquera, J., Adler, F.R.: Evolution of virulence: a unified framework for
  coinfection and superinfection.
\newblock Journal of Theoretical Biology \textbf{195}(3), 293--313 (1998)

\bibitem{Newman}
Newman, M.E.: Threshold effects for two pathogens spreading on a network.
\newblock Physical review letters \textbf{95}(10), 108701 (2005)

\bibitem{May3}
Nowak, M.A., May, R.M.: Superinfection and the evolution of parasite virulence.
\newblock Proceedings of the Royal Society of London. Series B: Biological
  Sciences \textbf{255}(1342), 81--89 (1994)

\bibitem{Zhou}
Zhou, J., Hethcote, H.W.: Population size dependent incidence in models for
  diseases without immunity.
\newblock J. Math. Biol. \textbf{32}(8), 809--834 (1994)

\end{thebibliography}

% Non-BibTeX users please use

\end{document}